\documentclass[a4paper,draft,10pt]{amsproc}
\usepackage{amssymb}
\usepackage{amscd} %% Package for commutative diagrams
\theoremstyle{plain}
 \newtheorem{theorem}{\bf{Theorem}}[section]
 
 \newtheorem{lemma}{\bf{Lemma}}[section]
 \newtheorem{corollary}{\bf{Corollary}}[section]
\theoremstyle{definition}

\theoremstyle{remark}
 
 \numberwithin{equation}{section}

\renewcommand{\leq}{\leqslant}
\renewcommand{\geq}{\geqslant}

\title[Weighted composition operators]
{Weighted composition, Volterra and Integral operators on Hardy Zygmund-type spaces}
%% Please use the newest classification -- 2010

\author[Hassanlou and Abbasi
]{\textsc{Mostafa Hassanlou$^{1*}$
 and Ebrahim Abbasi$^{2}$
 } \\
\textit{\footnotesize $^1$Engineering Faculty of Khoy, Urmia University of Technology, Urmia, Iran\\
m.hassanlou@urmia.ac.ir\\
$^2$Department of Mathematics, Mahabad Branch, Islamic Azad University, Mahabad, Iran\\
ebrahimabbasi81@gmail.com
}}
\thanks{$^*$Corresponding author}
\begin{document}

%{\begin{flushleft}\baselineskip9pt\scriptsize
%PUBLICATIONS DE L'INSTITUT MATH\'EMATIQUE\newline
%Nouvelle s\'erie, tome 87(101) (2010), od--do \hfill DOI:
%\end{flushleft}}
\vspace{18mm}
\setcounter{page}{1}
\thispagestyle{empty}

\begin{abstract}
In this paper, we investigate weighted composition, Volterra and Integral operators on second derivative Hardy spaces. Some equivalent conditions
for boundedness of the operators will be given using the boundedness on the Hardy spaces.  Also we give a criteria for
compactness of weighted composition operators.
\end{abstract}

\maketitle
%**************************************************************************************************
%**************************************************************************************************
%**************************************************************************************************
\section{\sc\bf Introduction}
Let $\mathbb{D}$ be the open unit disc in the complex plan $\mathbb{C}$ and $H(\mathbb{D})$ be the space of analytic function on $\mathbb{D}$. For
$\psi, \varphi \in H(\mathbb{D})$ with $\varphi(\mathbb{D}) \subset \mathbb{D}$ the weighted composition operators  $W_{\varphi, \psi}$
defined by
$$ W_{\varphi, \psi} f (z) = \psi(z) f(\varphi(z)), \ \ \ \  f \in  H(\mathbb{D}), z \in \mathbb{D},$$
which is a generalization of the well-known composition operators $C_{\varphi}$ and multiplication operators $M_{\psi}$. For $g \in H(\mathbb{D})$, the Volterra type operator $T_g$ is defined by
$$ T_{g} f(z) = \int_{0}^{z} f(w) g'(w) dw, \ \ \ \ f \in H(\mathbb{D}).  $$
Also the integral operator $I_g$ is defined as follows
$$ I_g f(z) = \int_{0}^{z}  f'(w) g(w) dw.$$
The space of bounded analytic functions on $\mathbb{D}$ is denoted by $H^\infty$ which is a Banach space with the norm $\|f\|_\infty=\sup_{z\in\mathbb{D}}|f(z)|.$
For $1\leq p<\infty,$ the Hardy space $H^p$ consists of all analytic functions $f\in H(\mathbb{D})$ for which
\[
 \|f\|_{H^p}^p= \sup_{0< r < 1} M_r(f,p)=\sup_{0< r < 1}\frac{1}{2\pi}\int_{0}^{2\pi}|f(re^{i \theta})|^p d\theta<\infty.
\]
These spaces are Banach with the norm $\|\cdot\|_{H^p}$.

 Let $1\leq p<\infty$. We denote by $S_2^p$, the  space of analytic functions on $\mathbb{D}$ such that   second derivative is in the Hardy space. So,
$$S_2^p = \{ f \in H(\mathbb{D}): f'' \in H^p  \}.$$
We define the norm on $S_2^p$ as follows
\[
\|f\|_{S_2^p}= |f(0)| + |f'(0)|+\|f''\|_{H^p},
\]
and equipped with this norm, $S_2^p$ is a Banach space. These spaces may be called Hardy Zygmund-type spaces.
In particular,  $S_1^p= S^p$, the space of analytic function with derivative in Hardy space, which is
investigated along with weighted composition operators and Volterra operators in \cite{con2, cai, mac1, roan}.

%Theorem 3.11 of   implies that if  $f' \in H^p$ then $f$ is continuous in the closed disk and absolutely continuous on the boundary. Then
%$S^p$ contained in $\mathcal{A}$, the disk algebra with the norm $||f||_{\mathcal{A}} = \sup_{z \in \mathbb{D}} |f(z)|$. The inclusion map
%$S^p \hookrightarrow \mathcal{A}$ is bounded due to the closed graph theorem and $||f||_{\mathcal{A}} \leq c ||f||_{S^p}$ for $f \in S^p$.
%($c$ may be just 1 ???????)
$S^p$ is a Banach space with the norm $\|f\|_{S^p}=|f(0)|+\|f'\|_{H^p}$. It is well known by Theorem 3.11 in \cite{duren} that, if $f \in S^1$, then $f$ extends continuously to $\overline{\mathbb{D}}$. Thus, the functions in $S^p$ belong to the disc algebra $A$ (the space of analytic
functions on $\mathbb{D}$ and continuous on $\overline{\mathbb{D}}$ endowed with the norm $||f||_A = \sup_{z \in \mathbb{D}} |f(z)|$).

It can be proved that $S_2^p \subset S^p \subset H^p$ and there exists a positive constant $C_p$ such that
$$||f||_{S^p} \leq C_p ||f||_{S_2^p}, \ \ \ \ f \in S_2^p.$$
For $1 \leq p \leq \infty $ and $f \in H^p$ we have
\[
|f(z)|\leq \frac{\|f\|_{H^p}}{(1-|z|^2)^{1/p}}.
\]
If $f \in S_2^p$ then there exists a positive constant $C$ such that
\begin{equation}\label{r21}
  |f^{(k)}(z)| \leq C \frac{\|f\|_{S_2^p}}{(1-|z|^2)^{k -1 + 1/p}},
\end{equation}
where $k=0, 1, 2$.

Roan in \cite{roan} characterized bounded, compact and isometric composition operators on $S^p$. MacCluer investigated composition operators on the space
$S^p$ in terms of Carleson measure, \cite{mac1}. A characterization for boundedenss, (weak) compactness and complete continuity of weighted composition
operators can be found in   \cite{con2}. Composition and multiplication operators with some different norms on $S^2$ were studied in \cite{cai}.  Also Volterra  type operators on $S^p$ spaces studied by authors of \cite{qin}.

In this paper, noting that the spaces introduced here is not known in the literature, we characterize boundedness of  weighted composition on
$S_2^p$ in terms of such operators on $H^p$ spaces. Also compact weighted composition operators will be studied.
Furthermore, we have a brief investigation on the bounded integral type operators on the $S_2^p$ spaces.

All positive constants are denoted by $C$ which may be varied from one place to another.
%************************************************************************************
%**************************************************************************************
%**************************************************************************************
%***********************************************************************************
\section{\sc\bf Boundedness of $W_{\varphi, \psi} :   S_2^p \rightarrow S_2^p$ }
In this section, some conditions for boundedenss of weighted composition operators form $S_2^p$ into $S_2^p$ or $H^p$ will be given.
Since
$$(W_{\varphi, \psi}f)'' = \psi'' f (\varphi) + (2 \psi' \varphi' + \psi \varphi'') f'(\varphi) + \psi \varphi'^2 f''(\varphi) $$
then the study of $W_{\varphi, \psi} $ on $S_2^p$ spaces is related to the study of the  following operators
\begin{align*}
 W_{\varphi, \psi''} : &  S_2^p \rightarrow H^p  \\
 W_{\varphi, 2 \psi' \varphi' + \psi \varphi''} : &  S^p \rightarrow H^p \\
 W_{\varphi, \psi \varphi'^2} : &  H^p \rightarrow H^p.
\end{align*}
Although the above mentioned operators can be used but without loss of generality we replace the operators from $S^p$ or $S_2^p$ by the operators
on $H^p$ spaces.
%***********************************************************************************
\begin{theorem}\label{th22} \cite{cuk1}
  Let $\psi$ be an analytic function on $\mathbb{D}$ and $\varphi$ be an analytic
self-map of $\mathbb{D}$. Let $0 < p \leq q < \infty$. Then the weighted composition operator $W_{\varphi, \psi}$ is
bounded from $H^p$ into $H^q$ if and only if
$$\sup_{a\in \mathbb{D}} \int_{\partial \mathbb{D}} \left ( \frac{1-|a|^2}{|1-\overline{a} \varphi (w)|^2} \right )^{q/p} |\psi(w)|^q d\sigma (w) < \infty,$$
where $\partial \mathbb{D}$ is the unit circle and $d \sigma$ is the normalized arc length measure on $\partial \mathbb{D}$.
\end{theorem}
%************************************************************************************
\begin{lemma} \label{L22}
  Let $1 \leq p < \infty$. Then $W_{\varphi, \psi} :   S \rightarrow H^p$ is bounded if and only if $\psi \in H^p$, where $S$ is $S_2^p$ or $S^p$.
\end{lemma}
\begin{proof}
  Suppose that $\psi \in H^p$ and $f \in S$. Then
  \begin{align*}
    ||W_{\varphi, \psi}f||_{H^p} = & || \psi f \circ \varphi ||_{H^p} \leq ||f||_{\infty} ||\psi||_{H^p}
    \leq C ||f||_{S} ||\psi||_{H^p}.
  \end{align*}
So $W_{\varphi, \psi}$ is bounded from $S$ into $H^p$. For converse, using the function $f(z) =1 \in S$ we have
$$ ||W_{\varphi, \psi}|| \geq ||W_{\varphi, \psi}(1)||_{H^p} = ||\psi||_{H^p}.  $$
\end{proof}
%*****************************************************************************************************
\begin{theorem} \label{th23}
Let $1 \leq p < \infty$. Then the following conditions are equivalent:
\begin{itemize}
  \item [(a)] $W_{\varphi, \psi} :   S_2^p \rightarrow S_2^p$ is bounded.
  \item [(b)] $\psi \in S_2^p$ and the operators
$  W_{\varphi, 2 \psi' \varphi' + \psi \varphi''} : H^p \rightarrow H^p$ and $W_{\varphi, \psi \varphi'^2} :  H^p \rightarrow H^p$ are bounded.
  \item [(c)] $\psi \in S_2^p$ and
\begin{align*}
\sup_{a\in \mathbb{D}} \int_{\partial \mathbb{D}}   \frac{1-|a|^2}{|1-\overline{a} \varphi (w)|^2}  |2 \psi'(w) \varphi'(w) + \psi(w) \varphi''(w)|^p d\sigma (w) < &  \infty, \\
\sup_{a\in \mathbb{D}} \int_{\partial \mathbb{D}} \frac{1-|a|^2}{|1-\overline{a} \varphi (w)|^2}  |\psi (w)\varphi'^2 (w)|^p d\sigma (w) < & \infty.
\end{align*}
\end{itemize}
\end{theorem}
\begin{proof}
\textbf{$(b) \rightarrow (a)$:} Suppose that $\psi \in S_2^p$ and
$  W_{\varphi, 2 \psi' \varphi' + \psi \varphi''} : H^p \rightarrow H^p$ and $W_{\varphi, \psi \varphi'^2} :  H^p \rightarrow H^p$ be bounded. Let
$f \in S_2^p$. Then using the proof of Lemma \ref{L22} we obtain
\begin{align} \label{r22}
||(\psi f \circ \varphi)''||_{H^p} \leq & || \psi'' f \circ \varphi ||_{H^p} + || (2 \psi' \varphi' + \psi \varphi'') f' \circ \varphi ||_{H^p}
+ || \psi \varphi'^2 f'' \circ \varphi||_{H^p} \nonumber \\
\leq & C_1 ||\psi''||_{H^p} ||f||_{S_2^p} + C_2 ||f'||_{H^p} + C_3  ||f''||_{H^p} \nonumber \\
\leq &  C_1 ||\psi''||_{H^p} ||f||_{S_2^p} + C_2 ||f||_{S^p} + C_3  ||f||_{S_2^p} \nonumber \\
\leq & C_1 ||\psi''||_{H^p} ||f||_{S_2^p} + C_2 C_p ||f||_{S_2^p} + C_3  ||f||_{S_2^p}.
\end{align}
Also we have
\begin{align} \label{r23}
|(\psi f \circ \varphi)(0)| + |(\psi f \circ \varphi)'(0)| = & |\psi(0) f(\varphi(0))| + |\psi'(0) f (\varphi(0)) + \psi(0) \varphi'(0) f'(\varphi(0))| \nonumber \\
\leq & C \frac{|\psi(0)| \|f\|_{S_2^p}}{(1-|\varphi(0)|^2)^{-1 + 1/p}} + C \frac{|\psi'(0)| \|f\|_{S_2^p}}{(1-|\varphi(0)|^2)^{-1 + 1/p}} \nonumber\\
& + C \frac{|\psi(0) \varphi'(0)| \|f\|_{S_2^p}}{(1-|\varphi(0)|^2)^{1/p}}.
\end{align}
From \eqref{r22} and \eqref{r23} we get the desired result.

\textbf{$(a) \rightarrow (b)$:} Suppose that $W_{\varphi, \psi} :   S_2^p \rightarrow S_2^p$ be a bounded operator. Consider the function $f(z) =1 \in S_2^p$. Then
$$ ||W_{\varphi, \psi}|| \geq ||W_{\varphi, \psi}(1)||_{S_2^p} = ||\psi||_{S_2^p}.$$
So $\psi \in S_2^p$. Consider the function $f(z) = z \in  S_2^p$. It can be easily obtained that $\psi \varphi \in S_2^p$ or
$\psi'' \varphi + 2 \psi' \varphi' + \psi \varphi'' \in H^p$. Therefore $2 \psi' \varphi' + \psi \varphi'' \in H^p$. Also replacing the function $f(z)= z^2$, one can get
$\psi \varphi'^2 \in H^p$.

Now we show that $W_{\varphi, \psi \varphi'^2} :  H^p \rightarrow H^p$ is bounded. Let $f \in H^p$. There exists a function $g \in S_2^p$ such that
$f = g''$. Hence
\begin{align*}
|| W_{\varphi, \psi \varphi'^2} f & ||_{H^p} =  ||\psi \varphi'^2 f \circ \varphi ||_{H^p} \\
=& ||\psi \varphi'^2 f \circ \varphi  + (2 \psi' \varphi' + \psi \varphi'') g' \circ \varphi + \psi'' g \circ \varphi
- (2 \psi' \varphi' + \psi \varphi'') g' \circ \varphi - \psi'' g \circ \varphi ||_{H^p} \\
\leq & || W_{\varphi, \psi} g ||_{S_2^p} +  ||(2 \psi' \varphi' + \psi \varphi'') g' \circ \varphi||_{H^p} + ||\psi'' g \circ \varphi ||_{H^p} \\
\leq & C_1 ||g||_{S_2^p} + C_2 ||2 \psi' \varphi' + \psi \varphi''||_{H^p} ||g'||_{\infty} + C_3  ||\psi''||_{H^p} ||g||_{\infty} \\
\leq & C_1 ||f||_{H^p} + C_2 ||2 \psi' \varphi' + \psi \varphi''||_{H^p} ||f||_{H^p} + C_3  ||\psi''||_{H^p} ||f||_{H^p}
\end{align*}
In a similar way we can prove that $  W_{\varphi, 2 \psi' \varphi' + \psi \varphi''} : H^p \rightarrow H^p$  is a bounded operator.

The equivalency of $(b)$ and $(c)$ comes from  Theorem \ref{th22}.
\end{proof}
%***********************************************************************************************
The following corollaries can be obtained from Theorem \ref{th23}.
\begin{corollary}
  Let $1 \leq p < \infty$. Then the following conditions are equivalent:
\begin{itemize}
  \item [(a)] $C_{\varphi} :   S_2^p \rightarrow S_2^p$ is bounded.
  \item [(b)] The operators
$W_{\varphi, \varphi''} : H^p \rightarrow H^p$ and $W_{\varphi, \varphi'^2} :  H^p \rightarrow H^p$ are bounded.
  \item [(c)]
  $$\sup_{a\in \mathbb{D}} \int_{\partial \mathbb{D}} \frac{1-|a|^2}{|1-\overline{a} \varphi (w)|^2}  |\varphi''(w)|^p d\sigma (w) < \infty,$$
  $$\sup_{a\in \mathbb{D}} \int_{\partial \mathbb{D}}  \frac{1-|a|^2}{|1-\overline{a} \varphi (w)|^2}  |\varphi'^2 (w)|^p d\sigma (w) < \infty,$$
\end{itemize}
\end{corollary}
%**************************************************************************************************
\begin{corollary}
  Let $1 \leq p < \infty$. Then the following conditions are equivalent:
\begin{itemize}
  \item [(a)] $M_{\psi} :   S_2^p \rightarrow S_2^p$ is bounded.
  \item [(b)] $\psi \in S_2^p$ and the operators
$M_{2 \psi'} : H^p \rightarrow H^p$ and $M_{\psi} :  H^p \rightarrow H^p$ are bounded.
  \item [(c)] $\psi \in S_2^p$ and
  $$\sup_{a\in \mathbb{D}} \int_{\partial \mathbb{D}} \frac{1-|a|^2}{|1-\overline{a} w|^2}  |2 \psi'(w)|^p d\sigma (w) < \infty,$$
  $$\sup_{a\in \mathbb{D}} \int_{\partial \mathbb{D}}  \frac{1-|a|^2}{|1-\overline{a} w|^2}  |\psi(w)|^p d\sigma (w) < \infty,$$
\end{itemize}
\end{corollary}
%************************************************************************************
%**************************************************************************************
%**************************************************************************************
%***********************************************************************************
\section{\sc\bf Compactness of $W_{\varphi, \psi} :   S_2^p \rightarrow S_2^p$ }
%**************************************************************************************
In this section, we investigate compactness of $W_{\varphi, \psi} :   S_2^p \rightarrow S_2^p$. Similar to the boundedness,
the compactness is also related to the compactness of the operators
\begin{align*}
 W_{\varphi, \psi''} : &  S_2^p \rightarrow H^p  \\
 W_{\varphi, 2 \psi' \varphi' + \psi \varphi''} : &  S^p \rightarrow H^p \\
 W_{\varphi, \psi \varphi'^2} : &  H^p \rightarrow H^p,
\end{align*}
and again we replace $S_2^p$ and $S^p$ by $H^p$. Contreras and Hern$\acute{a}$ndez-D$\acute{i}$az \cite{con1} proved that
the inclusion operator $j_p: S^p \hookrightarrow A $ is compact if and only if $1 < p \leq \infty$ and
the inclusion operator from $S^1$ into $H^1$ is compact. Since $S_2^p$ is a subspace of $S^p$, the inclusion operator $j_p: S_2^p \hookrightarrow A $ is compact if and only if $1 < p \leq \infty$ and
the inclusion operator from $S_2^1$ into $H^1$ is compact.
%for $(p, q) \not = (1, \infty)$ and $\psi \in H^q$, then $W_{\varphi, \psi}: S^p \rightarrow H^q$ is compact.
%*****************************************************************************************************
\begin{theorem}
Let $1 \leq p < \infty$. Then the following conditions are equivalent:
\begin{itemize}
  \item [(a)] $W_{\varphi, \psi} :   S_2^p \rightarrow S_2^p$ is compact.
  \item [(b)] The operators $W_{\varphi, \psi''}, W_{\varphi, 2 \psi' \varphi' + \psi \varphi''}, W_{\varphi, \psi \varphi'^2} :  H^p \rightarrow H^p$
 are compact.
\end{itemize}
\end{theorem}
\begin{proof}
\textbf{$(b) \rightarrow (a)$:} Assume that the operators in $(b)$ be compact and  $\{ f_n \}$ is a sequence in the unit ball of $S_2^p$ which converges to zero uniformly on compact subsets of $\mathbb{D}$. We show that $||W_{\varphi, \psi} f_n||_{S_2^p} \rightarrow 0$. Then
\begin{align} \label{r25}
||(\psi f_n \circ \varphi)''||_{H^p} \leq & || \psi'' f \circ \varphi ||_{H^p} + || (2 \psi' \varphi' + \psi \varphi'') f' \circ \varphi ||_{H^p}
+ || \psi \varphi'^2 f'' \circ \varphi||_{H^p} \nonumber \\
= & ||W_{\varphi, \psi''} f_n ||_{H^p} + ||W_{\varphi, 2 \psi' \varphi' + \psi \varphi''} f_n ||_{H^p} + ||W_{\varphi, \psi \varphi'^2 } f_n ||_{H^p} \rightarrow 0. \nonumber
\end{align}
On the other hand, since uniformly convergence on compact subsets implies pointwise convergence, so
$$ |(\psi f_n \circ \varphi)(0)| + |(\psi f_n \circ \varphi)'(0)| \rightarrow 0. $$
So
 $$||W_{\varphi, \psi} f_n||_{S_2^p} = |(\psi f_n \circ \varphi)(0)| + |(\psi f_n \circ \varphi)'(0)|  +
 ||(\psi f_n \circ \varphi)''||_{H^p}  \rightarrow 0. $$
\textbf{$(a) \rightarrow (b)$:} Suppose that $W_{\varphi, \psi} :   S_2^p \rightarrow S_2^p$ be compact. We just prove that $W_{\varphi, \psi \varphi'^2} :  H^p \rightarrow H^p$ is compact, the others are similar. Let $\{ f_n \}$ be a sequence in the unit ball of $H^p$ such that
which converges to zero uniformly on compact subsets of $\mathbb{D}$. There exists a sequence $\{ g_n \}$ in $S_2^p$ with $g_n'' = f_n$ and
$g_n \rightarrow 0$ uniformly on compact subsets of $\mathbb{D}$. So $|| W_{\varphi, \psi} g_n ||_{S_2^p} \rightarrow 0$. Hence
\begin{align*}
|| W_{\varphi, \psi \varphi'^2} f_n  ||_{H^p} = & ||\psi \varphi'^2 f_n \circ \varphi ||_{H^p} \\
\leq & || W_{\varphi, \psi} g_n ||_{S_2^p} +  ||(2 \psi' \varphi' + \psi \varphi'') g_n' \circ \varphi||_{H^p} + ||\psi'' g_n \circ \varphi ||_{H^p}.
\end{align*}
If $p>1$, noting that the inclusion operator $j_p: S_2^p, S^p \hookrightarrow A $ is compact, we have
$$ ||(2 \psi' \varphi' + \psi \varphi'') g_n' \circ \varphi||_{H^p} \leq || 2 \psi' \varphi' + \psi \varphi'' ||_{H^p} ||g_n'||_{A} \rightarrow 0 $$
and also
$$ ||\psi'' g_n \circ \varphi ||_{H^p} \leq || \psi'' ||_{H^p} ||g_n||_{A} \rightarrow 0. $$
If $p=1$, the compactness of the inclusion operator from $S_2^1$ into $H^1$  implies that $||g_n||_{H^1} \rightarrow 0$ and $||g_n'||_{H^1} \rightarrow 0$. So there exists a subsequence, say $\{ g_n \}$, such that $g_n (z) \rightarrow 0$ almost everywhere in $\mathbb{T}$. In particular, $|\psi'' (z) g_n (\varphi(z))| \rightarrow 0$ almost everywhere in $\mathbb{T}$. On the other hand
$$  |\psi'' (z) g_n (\varphi(z))| \leq ||g_n||_A |\psi''(z)| \leq
||j_1||_{S_2^1 \rightarrow H^1} ||g_n||_{S_2^1} |\psi''(z)| \leq
||j_1||_{S_2^1 \rightarrow H^1}  |\psi''(z)|, $$
for every $z \in \mathbb{T}$. So, by the dominated convergence theorem, $ ||\psi'' g_n \circ \varphi ||_{H^1} \rightarrow 0 $. In a similar way, we can prove that
$$ ||(2 \psi' \varphi' + \psi \varphi'') g_n' \circ \varphi||_{H^1}  \rightarrow 0. $$
From above equations, we can see that $|| W_{\varphi, \psi \varphi'^2} f_n  ||_{H^p} \rightarrow 0 $ and so $W_{\varphi, \psi \varphi'^2} :  H^p \rightarrow H^p$ is compact.
\end{proof}
For $p>1$, $\breve{C}$u$\breve{c}$kovi$\acute{c}$ and Zhao obtained that (among other results) if $1  < p \leq q < \infty$, then $W_{\varphi, \psi}: H^p\rightarrow H^q$ is compact if and only if
$$\limsup_{|a| \rightarrow 1} \int_{\partial \mathbb{D}} \left ( \frac{1-|a|^2}{|1-\overline{a} \varphi (w)|^2} \right )^{q/p} |\psi(w)|^q d\sigma (w) =0,$$
see \cite{cuk1}.
So, the conditions in the previous theorem are equivalent to
\begin{align*}
 \limsup_{|a| \rightarrow 1} \int_{\partial \mathbb{D}}  \frac{1-|a|^2}{|1-\overline{a} \varphi (w)|^2}  |\psi''(w)|^q d\sigma (w) = & 0 \\
 \limsup_{|a| \rightarrow 1} \int_{\partial \mathbb{D}} \frac{1-|a|^2}{|1-\overline{a} \varphi (w)|^2}  |2 \psi'(w) \varphi'(w) + \psi(w) \varphi''(w)|^q d\sigma (w) =  & 0 \\
 \limsup_{|a| \rightarrow 1} \int_{\partial \mathbb{D}}  \frac{1-|a|^2}{|1-\overline{a} \varphi (w)|^2}  |\psi(w) \varphi'^2 (w)|^q d\sigma (w) = & 0.
\end{align*}
%**************************************************************************************
%***********************************************************************************
\section{\sc\bf Boundedness of $T_g, I_g :   S_2^p \rightarrow S_2^p$ }
%***********************************************************************************
\begin{theorem}
  Let $1 \leq p < \infty$. Then $T_g:   S_2^p \rightarrow S_2^p$ is bounded if and only if $g \in S_2^p$.
\end{theorem}
\begin{proof}
Suppose that $g \in S_2^p$. For any $f \in S_2^p$ we have
\begin{align*}
  ||T_g f||_{S_2^p} = & ||f' g' + f g''||_{H^p} \\
  \leq & ||f' g'||_{H^p} + ||f g''||_{H^p} \\
  \leq & ||g'||_{H^p} || f' ||_{\infty} + || g'' ||_{H^p} ||f||_{\infty} \\
  \leq & 2C ||g||_{S_2^p} ||f ||_{S_2^p}
\end{align*}
For the converse, consider the constant function $f(z) =1 \in S_2^p$. Then $$||T_g|| \geq ||T_g f||_{S_2^p} = ||g||_{S_2^p} - |g(0)|.$$
\end{proof}
%****************************************************************************************************
\begin{theorem}
  Let $1 \leq p < \infty$. Then $I_g:   S_2^p \rightarrow S_2^p$ is bounded if and only if $g \in S^p$.
\end{theorem}
\begin{proof}
Suppose that $g \in S^p$. For any $f \in S_2^p$ we have
\begin{align*}
  ||I_g f||_{S_2^p} = & ||f'' g + f' g'||_{H^p} \\
  \leq & ||f'' g||_{H^p} + ||f' g'||_{H^p} \\
  \leq & ||f''||_{H^p} || g ||_{\infty} + || f' ||_{\infty} ||g'||_{H^p} \\
  \leq &  2 C ||f||_{S_2^p} ||g||_{S^p}
\end{align*}
For the converse, consider the constant function $f(z) =z \in S_2^p$. Then
$$||I_g|| \geq ||I_g f||_{S_2^p} = ||g||_{S^p}.$$
\end{proof}
We can change the condition $g \in S^p$ with $g \in S_2^p$ in the previous theorem.
%****************************************************************************************
\\
\\
\textbf{Acknowledgements}

The authors are very thankful to referees for reviewing the paper.
\newline
\\
\textbf{Funding}

Not applicable.
\\
\\
\textbf{Availability of data and materials}

Data sharing not applicable to this article.
\\
\\
\textbf{Declarations}
\\
\\
\textbf{Ethics approval and consent to participate}

Not applicable.
\\
\\
\textbf{Consent for publication}

Not applicable.
\\
\\
\textbf{Competing interests}

The authors declare that they have no competing interests.
\\
\\
\textbf{Authors’ contributions}

MH proposed the main idea of the paper and wrote the boundedness part. EA designed the compactness part of the operators. All authors read and approved the final manuscript.
%*******************************************************************************

%00000000000000000000000000000000000000000000000000000000000000000000000000000000000000000000000000000000000
%00000000000000000000000000000000000000000000000000000000000000000000000000000000000000000000000000000000000

\end{document}